\documentclass[12pt,reqno]{amsart}

\usepackage{amsfonts,amsmath,amsthm}
\usepackage{amssymb,epsfig}
\usepackage{cases}
\usepackage[usenames]{color}
\usepackage{enumerate} 
\usepackage{color} %color
\definecolor{vert}{rgb}{0,0.6,0}

\theoremstyle{plain}
\newtheorem{thm}{Theorem}[section]

\newtheorem{lem}[thm]{Lemma}

\newtheorem{prop}[thm]{Proposition}
\theoremstyle{remark}
\newtheorem{rem}{\bf{Remark}}
\numberwithin{equation}{section}
\newcommand{\E}{\mathbb{E}}

\newcommand{\N}{\mathbb{N}}

\newcommand{\Rr}{\mathbb{R}}
\newcommand{\R}{\mathbb{R}}

\newcommand{\T}{\mathbb{T}}
\newcommand{\Tt}{\mathbb{T}}

\newcommand{\cL}{\mathcal{L}}

\newcommand{\al}{\alpha}
\newcommand{\gam}{\gamma}

\newcommand{\ep}{\varepsilon}

\newcommand{\lam}{\lambda}

\newcommand{\Del}{\Delta}

\newcommand{\Lam}{\Lambda}

\newcommand{\Div}{{\rm div}\,}

\begin{document}
\title[Stationary mean field games with congestion and quadratic Hamiltonians]{Existence for stationary mean field games with quadratic Hamiltonians with congestion}

\author{Diogo A. Gomes}
\address[D. A. Gomes]{
King Abdullah University of Science and Technology (KAUST), CSMSE Division , Thuwal 23955-6900. Saudi Arabia, and  
KAUST SRI, Uncertainty Quantification Center in Computational Science and Engineering.}
\email{diogo.gomes@kaust.edu.sa}
\author{Hiroyoshi Mitake}
\address[H. Mitake]{
Institute for Sustainable Sciences and Development, 
Hiroshima University 1-4-1 Kagamiyama, Higashi-Hiroshima-shi 739-8527, Japan}
\email{hiroyoshi-mitake@hiroshima-u.ac.jp}

 %Existence problems: global existence, local existence, non-existence

\thanks{
DG was partially supported by
KAUST SRI, Uncertainty Quantification Center in Computational Science and Engineering, CAMGSD-LARSys through FCT-Portugal and by grants
PTDC/MAT-CAL/0749/2012.
HM was partially supported by JST program to disseminate tenure tracking system. 
}
\date{\today}

\begin{abstract}
In this paper, we investigate the existence and uniqueness of solutions
to a stationary mean field game model introduced 
by J.-M. Lasry and P.-L. Lions. 
This model features a quadratic Hamiltonian 
with possibly singular congestion effects. Thanks to a new class of a-priori bounds, 
combined with the continuation method, we prove the existence of smooth solutions in arbitrary
dimensions. 
\end{abstract}

\maketitle

%%%%%%%%%%%%%%%%%%%%%%%%%%%%%%%%%%%%%%%%%%%%%%%%%%%%%%%%%%%%%%%%%%%%%%%%%%%%%%%%%%%%%%%%%%%%%%%%%%%%%%%%%%%%%%%%%%%%%%%%%%%%%%%%%%%%%%%%%%%%%%%%%%

\section{Introduction}

Since the seminal papers \cite{C1, C2, ll1, ll2, ll3, ll4}, research on mean field games has been extremely active (see, for instance, the recent surveys \cite{llg2, cardaliaguet, achdou2013finite, GS} and the references therein).  Nevertheless, 
several fundamental questions have not yet been answered. 
%This paper is a part of a long-term program to understand the regularity properties of mean-field games and obtain new methods to address
%various difficulties that arise in real-world examples.  
%
In this paper,
we address one of those, and
prove existence and uniqueness of smooth solutions for stationary mean field games with congestion and quadratic Hamiltonian. 

Mean field games model 
large populations of rational agents who move according to certain stochastic optimal 
control goals.
To simplify the presentation, we
will work in the periodic setting, that is in the $d$ dimensional standard torus $\Tt^d$,  $d\geq 1$.
We consider a large stationary population of agents, whose 
statistical information is encoded in an unknown probability density $m:\Tt^d\times[0, +\infty)\to \Rr$. Each individual agent
wants to minimize an infinite horizon discounted cost given by
\[
u(x,t)=\inf\left\{ 
\E\left[\int_t^ {+\infty}e^{-s}\left(\frac{m(X(s),s)^\alpha|v(s)-b(X(s))|^2}{2}
+V(X(s), m(X(s),s))\right) \,ds\right]\right\}, 
\]
where the infimum is taken over all progressively measurable controls $v$,
\[
dX=v dt+\sqrt 2 d W_t \ \text{with} \ X(0)=x,
\]
$W_t$ and $\E$ denote a standard $d$-dimensional Brownian motion 
and the expected value, respectively.  
The constant $0<\alpha<1$ determines the strength of congestion effects
in the term 
$m^\alpha |v-b(x)|^2$, and makes it costly to move in areas of high density 
with a drift $v$ substantially different from a reference vector field $b:\Tt^d\to \Rr^d$.
The function $V:\Tt^d\times \Rr^+\to \Rr$ 
accounts for 
additional spatial preferences of the agents. We assume $b$ and $V$ to be smooth functions.  
%$V$ are given smooth functions, and we will give precise assumptions later. 
%The function $m$ is a unknown probability
%measure which is determined through the minimization problem itself. 
%The cost 

Under standard assumptions of rationality and symmetry, one can derive a mean field problem which models this setup. 
A detailed discussion can be found
in \cite{LCDF, ll3}, where the following
problem, consisting of a viscous Hamilton-Jacobi equation for $u$, coupled with a 
Fokker-Planck equation for $m$, was introduced:
\[
\begin{cases}
\displaystyle 
-u_t+u-\Del u + \frac{|Du|^2}{2m^{\al}}+b(x)\cdot Du= V(x,m)\\
m_t+m-\Del m-\Div\big(m^{1-\al}Du\big)-\Div(mb)=0, 
\end{cases}
\]
together with initial conditions for $m(x,0)$ and suitable asymptotic behavior for $u$. 

In the present paper, we consider a stationary version of this problem, which for
$u, m:\Tt^d\to \Rr$, $m>0$, is given by the system
\begin{numcases}
{}
u-\Del u + \frac{|Du|^2}{2m^{\al}}+b(x)\cdot Du= V(x,m)
& in $\T^{d}$, \label{mfg-1} \\
m-\Del m-\Div\big(m^{1-\al}Du\big)-\Div(mb)=1,  
& in $\T^d$.   \label{mfg-2} 
\end{numcases}
%where $Df=(\partial f/\partial x_1,\dots, \partial f/\partial x_n)$ and $\Del f=\sum_{i=1}^{d}f_{x_ix_i}$ for a smooth function $f$. 
where the right hand side of the second equation is an additional source term for $m$ (to avoid the trivial solution $m=0$).  In \cite{LCDF}, only the uniqueness of smooth solutions was proven. However, existence of solutions, the main result of this paper, was not yet known for both stationary and time-dependent problems. The fundamental difficulty lies in the possibly singular behavior due to congestion. The dependence on $m$ in  
the optimal control problem causes the singularity in the equation \eqref{mfg-1}, for which we had to develop a new class of estimates.  
%Therefore, as far as we understand, this problem cannot be addressed with the standard tools
%available for local mean-field games.
%
%We can easily see from the view of the optimal control problem that 
%the probability measure $m$ highly-sensitive influences the value function. 
%This causes the appearance of the singularity in the equation \eqref{mfg-1}. 
%there is no existence results for this model. 
%In this paper we prove a-priori estimates for $\|m^{-1}\|_{\infty}$, from which we obtain existence of classical solutions to \eqref{mfg-1}, \eqref{mfg-2}.
Thanks to those we obtain our main result, which is 
\begin{thm}\label{thm:main}
Assume the following{\rm:}
\begin{itemize}
\item[{\rm(A1)}] $0 \le\al<1${\rm;} 
\item[{\rm(A2)}] $V:\T^d\times \Rr^{+}\to \Rr$, 
$V(x,m)\in C^\infty(\T^d\times \Rr^+)$ is globally bounded with bounded derivatives and non-decreasing with respect to $m${\rm;} 
\item[{\rm(A3)}] $b:\Tt^d\to \Rr^d$, $b\in C^\infty(\Tt^d)$.
\end{itemize}
Then there exists a solution $(u,m)\in C^{\infty}(\T^d)\times C^{\infty}(\T^d)$ to \eqref{mfg-1}-\eqref{mfg-2} with $m>0$.  
Furthermore, if $V$ is strictly increasing with respect to $m$, then a solution is unique. 
\end{thm}

%The uniqueness of classical solutions to a time-dependent version of the  problem was established in \cite{LCDF} (as detailed in the note \cite{GueU}).
%However, as far as the authors know, no existence results were known for this model.  
%The key difficulty consists in controlling the term $m^{\alpha}$ in the denominator in the Hamiltonian in \eqref{mfg-1}. 
Numerous a-priori bounds for mean field games 
have been proved by various authors, including the first author (see, for instance  \cite{ll1}, \cite{ll2}, \cite{ll3}, \cite{GSM},  \cite{GIMY},  \cite{CLLP},
\cite{GPM1}, \cite{GPM2}, \cite{GPM3} \cite{GPatVrt}, \cite{porretta}, \cite{P2}).    
However, these bounds were designed to address a different
coupling, namely mean field games where the local dependence on $m$ is not singular when $m=0$. A typical example
is the following system
\begin{equation}\label{earon}
\begin{cases}
\displaystyle
u-\Del u + \frac{|Du|^2}{2} =m^\alpha
& \text{in}\,  \T^{d}\\
m-\Delta m-\Div(m Du)=1& \text{in}\, \T^d.  
\end{cases}
\end{equation}
In \eqref{earon}, the main difficulties are caused by the growth of the nonlinearity $m$, especially for large $\alpha>0$,
rather than singularities caused by $m$ vanishing.
Besides, \eqref{earon} can be regarded as an Euler-Lagrange equation of a suitable functional, whereas \eqref{mfg-1}-\eqref{mfg-2} does not have this structure. 

In Section \ref{sec:ape}, we start by exploring the
special form of \eqref{mfg-1} and \eqref{mfg-2} to obtain a bound
for $\|m^{-1} \|_{L^\infty(\T^n)}$. 
This estimate, combined with the techniques from \cite{CrAm}, yields a-priori regularity in $W^{2,p}(\T^d)$ for any $p\ge1$. From this, a simple argument shows
that any solution to  \eqref{mfg-1}-\eqref{mfg-2} is bounded in any Sobolev space $W^{k,p}(\T^d)$. Then, 
in Section \ref{sec:exist}, we prove the existence of solutions to \eqref{mfg-1}-\eqref{mfg-2} by using the continuation method together with the aforementioned
a-priori estimates. 
In Appendix \ref{sec:unique}, for completeness, we present the
uniqueness proof for 
solutions to \eqref{mfg-1}-\eqref{mfg-2}, based upon the ideas in \cite{LCDF} (see also \cite{GueU}). 
In a forthcoming paper \cite{GMi}, we will study general mean field games with congestion,  
for which the techniques of the present paper cannot be applied directly, as discussed in Remark \ref{rem:generalization} at the end of the next section.

%\bigskip
%\noindent
%{\bf REMARK FOR INTRO} 
%\begin{enumerate}
%\item Need to check the latest references and coordinate References. 
%For instance, we should be careful for the way to handle the note of Gueant, 
%since it is not published from any journal and moreover one cannot download it from his web any more. 
%So it seems to me not professional to cite this. On the other hand, the argument for the uniqueness is definitely affected by his note. This is a generalization of the paper of %Lasry-Lions \cite{ll3} but it is not so trivial. 
%We should find one of his paper which includes that idea. 
%\end{enumerate}

\section{A-priori Estimates}\label{sec:ape}

In this section, we obtain a-priori bounds for solutions of
 \eqref{mfg-1}-\eqref{mfg-2}. In particular, 
we prove a $L^\infty$ bound for $m^{-1}$.  
From this, we derive estimates for $u, m$ in 
$W^{2,p}(\T^d)$ for any $p\ge1$. Then, by a bootstrapping argument, we establish
smoothness of solutions.   

\begin{prop}\label{prop:bound}
There exists a constant $C:=C(\|V\|_{\infty})\ge0$ such that for any classical solution $(u, m)$ of \eqref{mfg-1}-\eqref{mfg-2}
we have $\left\|u\right\|_{L^\infty(\T^d)}\leq C$. Furthermore, 
$m\ge0$ on $\T^d$, and $\|m\|_{L^{1}(\T^d)}=1$. 
\end{prop}
\begin{proof}
The $L^\infty$ bound is obtained by evaluating the equation at points of maximum  of $u$ (resp., minimum)
and using the facts that at those points $Du=0$, $\Delta u\leq 0$ (resp.,  $\geq 0$), and $V$ is bounded on $\T^d\times[0,\infty)$.
We then observe that 
$m$ is non-negative by the maximum principle. Moreover,  
it has total mass $1$ by integrating \eqref{mfg-2}.
\end{proof}

\begin{prop}\label{prop:estimate1}
There exists a constant $C:=C(\|b\|_{\infty}, \|V\|_{\infty})\ge0$ such that for any classical solution $(u, m)$ of \eqref{mfg-1}-\eqref{mfg-2}
we have
\[
\left\|\frac{1}{m}\right\|_{L^\infty(\T^d)}\leq C. 
\]
\end{prop}
\begin{proof}
Let $r>\alpha$. 
Subtract equation \eqref{mfg-2} divided by $(r+1-\al)m^{r+1-\al}$ 
from  equation \eqref{mfg-1} divided by $rm^r$. 
Then, 
\begin{align}
\label{magic}
&\int_{\T^d}
\Big[
u-\Del u + \frac{|Du|^2}{2m^{\al}}+b\cdot Du-V
\Big]\cdot \frac{1}{rm^r}\,dx\\\notag
-&\, 
\int_{\T^d}
\Big[
m-\Del m-\Div\big(m^{1-\al}Du\big)-\Div(mb)
\Big]\cdot\frac{1}{(r+1-\al)m^{r+1-\al}}\,dx\\\notag
&
=
-\int_{\T^d}\frac{1}{(r+1-\al)m^{r+1-\al}}\,dx. 
\end{align}
Next, observe that  
\[
\int_{\T^d}\frac{\Del u }{rm^r}\,dx
=
\int_{\T^d}\frac{Du\cdot Dm }{m^{r+1}}\,dx,
\]
and
\[
\int_{\T^d}\frac{\Div\big(m^{1-\al}Du\big)}{(r+1-\al)m^{r+1-\al}}\,dx
=
\int_{\T^d}\frac{Du\cdot Dm }{m^{r+1}}\,dx.  
\]
Hence
\begin{equation}\label{eq:cancel}
\int_{\T^d}\frac{\Del u }{rm^r}\,dx
-\int_{\T^d}\frac{\Div\big(m^{1-\al}Du\big)}{(r+1-\al)m^{r+1-\al}}\,dx
=0. 
\end{equation}
Also, note the identity
\begin{align*}
&\int \Div(b m) \frac{m^{-r-1+\alpha}}{r+1-\alpha}
=
\int   m^{-r-1+\alpha} b \cdot Dm \\
&=
-\int b\cdot D \left(\frac{m^{-r+\alpha}}{r-\alpha}\right)=
\frac{1}{r-\alpha} \int \Div(b) m^{-r+\alpha}.
\end{align*}

Then \eqref{magic} is reduced to 
\begin{align*}
&\int_{\T^d}\frac{1}{(r+1-\al)m^{r+1-\al}}\,dx
+
\int_{\T^d}\frac{|Du|^2}{2rm^{r+\al}}\,dx
+
\int_{\T^d}\frac{|Dm|^2}{m^{r+2-\al}}\,dx\\
=&\, 
\int_{\T^d} \left[-\frac{V}{rm^{r}}-\frac{u}{rm^r}+\frac{1}{(r+1-\al)m^{r-\al}}
-\frac{b\cdot Du}{r m^r} -\frac{1}{r-\alpha}  \Div(b) m^{-r+\alpha}\right]\,dx\\
\le &\, 
\int_{\T^d}\frac{C}{rm^{r}}\,dx
+\int_{\T^d}\frac{C}{(r-\al)m^{r-\al}}\,dx+
\int_{\T^d}\frac{|b|^2}{rm^{r-\al}}+\frac{|Du|^2}{4 r  m^{r+\al}}\,dx
\end{align*}
in view of Proposition \ref{prop:bound}. 
%Fix $r_0>\alpha$.  
Consequently,
\begin{align*}
&\int_{\T^d}\frac{1}{(r+1-\al)m^{r+1-\al}}\,dx
+
\int_{\T^d}\frac{|Du|^2}{4rm^{r+\al}}\,dx
+
\int_{\T^d}\frac{|Dm|^2}{m^{r+2-\al}}\,dx\\
\le &\, 
\int_{\T^d}\frac{C}{rm^{r}}\,dx
+\int_{\T^d}\frac{C}{(r-\al) m^{r-\al}}\,dx.
\end{align*}
By Young's inequality, for $\al\in[0,1)$, we have
\[
\frac{C}{rm^r}\le \frac{1}{4(r+1-\al)m^{r+1-\al}}+C_r^1,
\]
and
\[ 
\frac{C}{(r-\al)m^{r-\al}}\le 
\frac{1}{4(r-\al)m^{r+1-\al}}+C_r^2, 
\]
with 
\[
C_r^{1}:=\frac{(1-\al)4^{\frac{r}{1-\al}}C^{\frac{r+1-\al}{1-\al}}}{r(r+1-\al)}, \ 
C_r^{2}:=\frac{4^{r-\al}C^{r+1-\al}(r-\alpha)^{r-\alpha-1}}{(r+1-\al)^{r+1-\alpha}}. 
\]
Therefore, 
\[
\frac{1}{r+1-\al}\int_{\T^d}\frac{1}{m^{r-\al+1}}
\le 
2(C_r^1+C_r^2). 
\]

Thus, we get 
\[
\left\|\frac{1}{m}\right\|_{L^{r+1-\al}(\T^d)}
\le
\Big[
2 (r+1-\al) (C_r^1+C_r^2)
\Big]^{\frac{1}{r+1-\al}}=:C_{\al}(r). 
\]
We can easily check  that, for any $r_0>\alpha$ there exists $C_\al$  for which 
\[
C_{\al}(r)\le C_{\al} \ \text{for all} \ r\in[r_0,\infty).
\qedhere
\]
%where $C_{\al}$ is independent of $r$. 
\end{proof}

\begin{prop}\label{prop:estimate3}
For any $p\geq 1$
there exists a constant $C:=C_p(\|b\|_{\infty}, \|V\|_{\infty})>0$ such that for any classical solution $(u, m)$ of \eqref{mfg-1}-\eqref{mfg-2}, we have 
$\|u\|_{W^{2,p}(\T^d)}+\|m\|_{W^{2,p}(\T^d)}\le C$. 
\end{prop}
\begin{proof}
Let $(u,m)$ be a classical solution $(u,m)$ to \eqref{mfg-1}-\eqref{mfg-2}.  
In view of Lemma \cite[Lemma 4]{CrAm}, combined with Proposition \ref{prop:estimate1}
%, and uniqueness of solutions, 
%any strong solution $(u,m)$ of \eqref{mfg-1}-\eqref{mfg-2} satisfies 
we conclude
that for all $p\in[1,\infty)$ there exists $C=C(\|V\|_{\infty},\|b\|_{\infty}, p)$ such that 
\[
\|u\|_{W^{2,p}(\T^d)}\le C.  
\]
In light of the Sobolev embedding theorem, we get 
\begin{equation}\label{estimate-holder}
\|u\|_{C^{1,\gam}(\T^d)}\le C\|u\|_{W^{2,p}(\T^d)}\le C.  
\end{equation}

Then, 
multiplying \eqref{mfg-2} by $m^{p}$ and using Young's inequality yield 
\begin{align*}
&
\int_{\T^d}m^{p+1}\,dx
+
p\int_{\T^d}m^{p-1}|Dm|^{2}\,dx
=
\int_{\T^d}m^{p}\,dx-p\int_{\T^d}m^{p}(g+b)\cdot Dm\,dx\\
\le&\, 
\left[
\frac{1}{2}\int_{\T^d}m^{p+1}\,dx+C
\right]
+
\left[
\frac{p}{2}\int_{\T^d}m^{p-1}|Dm|^{2}\,dx
+
Cp\int_{\T^d}(|g|^2+|b|^2)m^{p+1}\,dx
\right], 
\end{align*}
where $g:=Du/m^{\al}$. 
Noting that $|g|\le C$ in view of \eqref{estimate-holder} and 
$m^{p-1}|Dm|^2=C_p|Dm^{(p+1)/2}|^2$,   
we get 
\begin{equation}\label{ineq1}
\int_{\T^d}m^{p+1}\,dx
+
C_{p}\int_{\T^d}|Dm^{(p+1)/2}|^2\,dx
\le 
C+C_{p}^{'}\int_{\T^d}m^{p+1}\,dx. 
\end{equation}
Using H\"older's inequality we have
\[
\left(\int_{\T^d}m^{p+1}\right)^{1/(p+1)}
\leq \left(\int_{\T^d} m \right)^{2/(2+d p)}\left(\int_{\Tt^d} m^{2^*(p+1)/2 }\right)^{\frac{d p/(2+d p)}{2^*(p+1)/2}}.
\]
Moreover, using the Sobolev embedding theorem, we get 
\begin{align*}
\int_{\T^d}m^{p+1}&\leq \left(\int_{\Tt^d} m^{2^*(p+1)/2 }\right)^{\frac{d p/(2+d p)}{2^*/2}}\\
&
\leq 
C
\left(\int_{\T^d}m^{p+1}\,dx
+
\int_{\T^d}|Dm^{(p+1)/2}|^2\,dx\right)^{d p/(2+d p)}.
\end{align*}
Then, using the previous estimate and the fact that $d p/(2+d p)<1$ in the right-hand side of \eqref{ineq1}, 
we conclude that 
\begin{equation}\label{estimate-Dm}
\int_{\T^d}m^{p+1}\,dx
+
\int_{\T^d}|Dm^{(p+1)/2}|^2\,dx\leq C.
\end{equation}

Note now that if $m\in W^{1,q}(\T^d)$, we have
\begin{equation}\label{eq:m}
m-\Delta m=m^{1-\alpha}\Delta u + (1-\al)m^{-\al}Du\cdot Dm+\Div(mb)+1
\in L^{q}(\T^d). 
\end{equation}
Thus,  by standard elliptic regularity $m\in W^{2, q}(\T^d)$.
Consequently $m\in W^{1, q^*}(\T^d)$. 
In light of \eqref{estimate-Dm} for $p=1$ we have $m\in W^{1,2}(\T^d)$.  
Thus we obtain $m\in W^{2,2}(\T^d)$  and $m\in W^{1,2^{\ast}}(\T^d)$. 
By iterating this argument, we finally get 
$m\in W^{2, q}(\T^d)$ for any $q<\infty$. 
\end{proof}

\begin{prop}\label{prop:bootstrap}
For any integer $k\geq 0$
there exists a constant $C:=C(\|b\|_{\infty}, \|V\|_{\infty}, k)>0$ such that 
any classical solution $(u,m)$ of \eqref{mfg-1}-\eqref{mfg-2} satisfies
$\|u\|_{W^{k,\infty}(\T^d)}+\|m\|_{W^{k,\infty}(\T^d)}\le C$. 
\end{prop}
\begin{proof}
Note that  $D(m^{-\al})=m^{-(1+\al)} Dm\in L^{p}(\T^d)$ for large $p>1$ 
in view of Propositions  \ref{prop:estimate1}, \ref{prop:estimate3}. This
implies $m^{-\al}\in W^{1,p}(\T^d)$.  Hence, by Morrey's theorem, 
we have $m^{-\al}\in C^{\gam}(\T^d)$ for some $\gam\in(0,1)$. 
Also note that $|Du|^2\in C^{\gam}(\T^d)$. 
Therefore, going back to the equation \eqref{mfg-1} we have
\begin{equation}\label{eq:u}
u-\Del u=-\frac{|Du|^2}{2m^{\al}}-b\cdot Du+V\in C^{\gam}(\T^d).   
\end{equation}
Then, in view of the elliptic regularity theory  
we get $u\in C^{2,\gam}(\T^d)$.   
Note that the norm in $C^{\gam}(\T^d)$ of the right hand side of \eqref{eq:u} is estimated by a constant which only depends on $\|b\|_{\infty}, \|V\|_{\infty}$. Thus,  $u\le C(\|b\|_{\infty},\|V\|_\infty)$. 
Next, going back to the equation \eqref{eq:m} for $m$
and noting that the right hand side is $C^{\gam}(\T^d)$ now, 
we get $m\in C^{2,\gam}(\T^d)$ with $m\le C(\|b\|_{\infty},\|V\|_\infty)$.  

Once we know that $u,m\in C^{2,\gam}(\T^d)$, 
\eqref{eq:u} and \eqref{eq:m} imply $u, m\in C^{3,\gam}(\T^d)$. 
By continuing this so-called ``bootstrap" argument, we get the conclusion.  \end{proof}

\begin{rem}\label{rem:generalization}
The methods in this section rely strongly on the particular structure of the Hamiltonian which allows for 
the cancellation in \eqref{eq:cancel}. 
In the forthcoming paper \cite{GMi}, we address mean field games which generalize \eqref{mfg-1}-\eqref{mfg-2}, 
namely of the form
\begin{numcases}
{}
u-\Del u + m^{\al}H(x,\frac{Du}{m^\alpha}) =V(x,m)
& in $\T^{d}$ \nonumber \\
m-\Del m-\Div\big(m D_pH(x,\frac{Du}{m^\al})\big)=1,  
& in $\T^d$,  
\nonumber
\end{numcases}
but for which the cancellation in \eqref{eq:cancel} no longer holds. 
Such problems are natural in many applications, for instance \eqref{eq:cancel} 
is not valid for
the anisotropic quadratic Hamiltonian
$H(x,p)=2^{-1}|A(x)p+b(x)|^2$, where $A$ is a strictly positive definite matrix. 
This is addressed  in \cite{GMi} with a different approach. 
\end{rem}

\section{Existence by Continuation Method}
\label{sec:exist}

In this section we prove the existence of a unique classical solution to \eqref{mfg-1}-\eqref{mfg-2} by using the continuation method. 
We work under the assumptions of Theorem \ref{thm:main}.
For $0\leq \lambda\leq 1$ we consider the problem 
\begin{equation}\label{conti-1}
\begin{cases}
\displaystyle
u_{\lam}-\Del u_{\lam} + \frac{|Du_{\lam}|^2}{2m^\alpha_{\lam}}+\lambda b(x) \cdot Du_\lam-\lam V(x, m)-(1-\lam) V_0(m)= 0
& \text{in} \ \T^d,\\
m_{\lam}-\Del m_{\lam}-\Div\big(m_{\lam}^{1-\alpha} Du_{\lam}\big)-\lam\Div(b m_\lam)=1   
& \text{in} \ \T^d,  
\end{cases}
\end{equation}
where $V_0(m):=\arctan(m)$. 
Let $E^{k}:=H^k(\T^d)\times H^k(\T^d)$ for $k\in\N$, and 
$E^{0}:=L^2(\T^d)\times L^2(\T^d)$. 

For any $k_0\in\N$ with $k_0>d/2$, 
we define the map
$F:[0,1]\times E^{k_0+2}\to E^{k_0}$ by 
\[
F(\lam,u,m):=
\left(\begin{array}{l}
\displaystyle
u-\Del u+\frac{|Du|^2}{2m^{\al}}+\lambda b(x) \cdot Du-\lam V(x,m)-(1-\lam)V_0(m)\\
m-\Del m-\Div\big(m^{1-\al}Du\big)-\lambda \Div ( b m )-1.
\end{array}
\right).  
\]
Then, we can rewrite \eqref{conti-1} as 
\[
F(\lam,u_\lam,m_\lam)=0. 
\]
Note that for any $\gamma>0$, 
the map $F$ is $C^\infty$ in the set $\{(u,m)\in E^{k_0+2}(\T^d), m>\gamma\}$. 
This is because for $k_0>d/2$, the Sobolev space $H^{k_0}(\T^d)$ is an algebra. 
Moreover, 
if $k_0$ is large enough, then any solution $(u_\lambda, m_\lambda)$ in  $E^{k_0+2}$ 
is, in fact, in $E^{k+2}$ for all $k\in\N$, by the a-priori bounds 
in Section \ref{sec:ape}.

We define the set $\Lam$ by 
\[
\Lam:=\left\{\lam\in[0,1]\mid \eqref{conti-1} \ \text{has a classical solution}
\ (u,m)\in E^{k_0+2}
\right\}. 
\]
When $\lambda=0$ we have an explicit solution, namely $(u_0, m_0)=(\pi/4, 1)$.
Therefore, $\Lam\not=\emptyset$. The main goal of this section is to prove 
\[
\Lam=[0,1]. 
\]
To prove this, we show that $\Lam$ is relatively closed and open
on $[0,1]$. 

The closeness of $\Lam$ is a straightforward consequence of the estimates in Section \ref{sec:ape}: 
\begin{prop}\label{prop:3-1}
The set $\Lambda$ is closed. 
\end{prop}
\begin{proof}
To prove that $\Lam$ is closed, we need to check that 
for any sequence $\lam_n\in\Lam$ such that $\lam_n\to\lam_0$ as $n\to\infty$, we have $\lambda_0\in \Lam$.  
Fix such a sequence and corresponding solutions $(u_{\lam_n}, m_{\lam_n})$ to \eqref{conti-1} with $\lam=\lam_n$.
Since the a-priori bound in Proposition \ref{prop:bootstrap} is independent of 
$n\in\N$, by taking a subsequence, if necessary, we may 
assume that 
$(u_{\lam_n}, m_{\lam_n})\to (u, m)$ in $E^{k_0+2}$.   
Moreover,  $m^{-1}_{\lam_n}\to m^{-1}$ in $C(\T^d)$. 
Therefore, we can take the limit in  \eqref{conti-1}, and we 
get that $(u,m)$ is solution to \eqref{conti-1} with $\lam=\lam_0$. This 
implies $\lam_0\in\Lam$. 
\end{proof}

To prove that $\Lam$ is relatively open in $[0,1]$, we need to check 
that for any $\lambda_0\in \Lam$ there exists a neighborhood of
 $\lambda_0$ contained in $\Lam$. To do so, we will use the implicit function theorem  (see, for example, \cite{D1}, chapter X).  
For a fixed $\lam_0\in\Lam$, we consider  
the Fr\'echet derivative $\cL_{\lam_0}:E^{k_0+2}\to E^{k_0}$ of
$(u,m)\mapsto F(\lambda_0, u, m)$ at the point $(u_{\lam_0},m_{\lam_0})$, 
which is given by 
\begin{align}
&\cL_{\lam_0}(v,f)\nonumber\\
=&
\, 
\left(
\begin{array}{l}
\displaystyle
v-\Del v
+\frac{Du_{\lam_0}\cdot Dv}{m_{\lam_0}^{\al}}
-\frac{\al |Du_{\lam_0}|^2 f}{2m_{\lam_0}^{\al+1}}
+{\lam_0}b\cdot Dv-\big(\lam_0D_mV+(1-\lam_0)D_mV_0\big)f\\
f-\Del f
-\Div\big(m_{\lam_0}^{1-\al}Dv\big)
-(1-\al)\Div\big(m_{\lam_0}^{-\al}f Du_{\lam_0}\big)-{\lam_0} \Div( b f)
\end{array}
\right).   
\label{frechet}
\end{align}
Because of the a-priori bounds for $u$ and $m$ in Section \ref{sec:ape}, 
we can extend the domain of $\cL_{\lam_{0}}$  
by continuity to $E^{k+2}$ for any $k\leq k_0$. 
%We write $\cL_{\lam_0}$ for $\cL[u_{\lam_0}, m_{\lam_0}]$ 
%for a fixed $(u_{\lam_0}, m_{\lam_0})$ by abuse of notation. 
We will prove that $\cL_{\lam_0}$ is an isomorphism from 
$E^{k+2}$ to $E^{k}$ for any $k\geq 0$.
%Before doing so we need several auxiliary results. 

Define 
the bilinear mapping 
$B_{\lam_0}[w_1,w_2]:E^{1}\to\R$ by 
\begin{align*}
&B_{\lam_0}[w_1,w_2]\\
:=&\, 
\int_{\T^d}
\left[v_1
+\frac{Du_{\lam_0}\cdot Dv_1}{m_{\lam_0}^{\al}}
-\frac{\al |Du_{\lam_0}|^2 f_1}{2m_{\lam_0}^{\al+1}}
+{\lam_0}b\cdot Dv_1-\big(\lam_0D_mV+(1-\lam_0)D_mV_0\big)f_1\right]f_2\\
&+Dv_1\cdot Df_2-m_{\lam_0}^{1-\al}Dv_1Dv_2\\
&\,
+
\left[f_1
-(1-\al)\Div\big(m_{\lam_0}^{-\al}f_1 Du_{\lam_0}\big)-{\lam_0} \Div( b f_1)
\right](-v_2)
-Df_1\cdot Dv_2\,dx. 
\end{align*}
We set $Pw:=(f, -v)$ for $w=(v,f)$, and 
we observe that if $w_1\in E^k$ with $k\ge2$, then 
\[
B_{\lam_0}[w_1,w_2]=\int_{\T^d}\cL_{\lam_0}(w_1) \cdot P w_2\,dx.   
\]
The boundedness of $B_{\lam_0}$ is a straightforward result of 
Proposition \ref{prop:bootstrap}:  
\begin{lem}\label{lem:3-2}
There exists a constant $C>0$ such that 
\[
|B_{\lam_0}[w_1, w_2]|
\leq C\|w_1\|_{E^{1}} \|w_2\|_{E^{1}}.  
\]
for any $w_1, w_2\in E^{1}$. 
\end{lem}
Thus, in view of the Riesz representation theorem for Hilbert spaces, 
there exists a linear mapping 
$A:E^{1}\to E^{1}$ 
such that 
\[
B_{\lam_0}[w_1, w_2]=(Aw_1, w_2)_{E^{1}}.  
\]
\begin{lem}\label{lem:3-3}
The operator $A$ is injective. 
\end{lem}
\begin{proof}
Let $w=(v, f)$. 
By Young's inequality we have
\begin{align*}
B_{\lam_0}[w, w]&=
\int_{\T^d} 
 \frac{\alpha Du_{\lam_0}\cdot Dv}{m_{\lam_0}^{\al}}f
-\frac{\al |Du_{\lam_0}|^2}{2m_{\lam_0}^{\al+1}}f^2
-\big(\lam_0D_mV+(1-\lam_0)D_mV_0\big) f^2 
-m_{\lam_0}^{1-\al}|Dv|^2\,dx\\
&
\leq
\int_{\T^d} 
-\big(\lam_0D_mV+(1-\lam_0)D_mV_0\big) f^2
+\frac{(\al-2) m_{\lam_0}^{1-\al}|Dv|^2}{2}
\\
&
\leq -C_{\lam_0}(\|Dv\|_{L^2(\T^d)}^2 +\|f\|_{L^2(\T^d)}^2) 
\end{align*}
for a constant $C_{\lam_0}$ which depends on bounds for $m_{\lam_0}$, and $Du_{\lam_0}$, but it is strictly positive
for any solution to \eqref{conti-1} since $0\le\al<1$. 
We have used Assumption (A1) and 
the strict monotonicity of $V$ on $m$. 
This implies that if $Aw=0$ we have $w=(\mu,0)$, for some constant $\mu$. 
Then, by computing 
\[0=(Aw, (0,\mu))=B[(\mu,0), (0,\mu)]=\mu^2,\] 
we conclude that $\mu=0$. 
\end{proof}

\begin{rem}
Note that the injectivity of the operator $A$ holds for all $0\le\alpha<2$. However, the a-priori estimates 
of the previous section are only valid for $0\le\alpha<1$. 
\end{rem}

\begin{lem}\label{lem:3-4}
The range $R(A)$ is closed, and $R(A)=E^{1}$. 
\end{lem}
\begin{proof}
Take a Cauchy sequence $z_n$ in the range of $A$, 
that is $z_n=Aw_n$, for some sequence $w_n=(v_n, f_n)$ . 
We claim that $w_n$ is a Cauchy sequence. 
We have
\begin{align*}
(z_n-z_m, w_n-w_m)_{E^1}
&=(A(w_n-w_m),w_n-w_m )_{E^1}\\
&\leq -C_{\lam_0}(\|D(v_n-v_m)\|_{L^2(\T^d)}^2+\|f_n-f_m\|_{L^2(\T^d)}^2). 
\end{align*}
Note that 
\begin{align*}
&|(z_n-z_m, w_n-w_m)_{E^1}|\\
\le&\, 
\|z_n-z_m\|_{E^{0}}\|w_n-w_m\|_{E^{0}}
+\|D(z_n-z_m)\|_{E^{0}}\|D(w_n-w_m)\|_{E^{0}}\\
\le&\, 
\|z_n-z_m\|_{E^{0}}\left(\|v_n-v_m\|_{L^{2}(\T^d)}
+\|f_n-f_m\|_{L^{2}(\T^d)}\right)\\
&+\|D(z_n-z_m)\|_{E^{0}}\left(\|D(v_n-v_m)\|_{L^{2}(\T^d)}
+\|D(f_n-f_m)\|_{L^{2}(\T^d)}\right)\\
\le&\, 
\ep\left(\|f_n-f_m\|^2_{L^{2}(\T^d)}+\|D(v_n-v_m)\|^2_{L^{2}(\T^d)}\right)
+C_{\ep}\|z_n-z_m\|^2_{E^{1}}\\
&+\|z_n-z_m\|_{E^{1}}\left(\|v_n-v_m\|_{L^{2}(\T^d)}
+\|D(f_n-f_m)\|_{L^{2}(\T^d)}\right)
\end{align*}
for $\ep>0$. 

If we fix a suitable small $\ep$ and combine the inequalities above, 
then we obtain  
\begin{align}
&\|D(v_n-v_m)\|_{L^2(\T^d)}^2+\|f_n-f_m\|_{L^2(\T^d)}^2\nonumber\\
\leq&\, 
C\|z_n-z_m\|_{E^1}^2+
C \|z_n-z_m\|_{E^1} (\|v_n-v_m\|_{L^2(\T^d)}+\|D(f_n-f_m)\|_{L^2}(\T^d))\nonumber\\
\leq&\, 
C\|z_n-z_m\|_{E^1}^2+
C \|z_n-z_m\|_{E^1} \|w_n-w_m\|_{E^1}. 
\label{estimate111}
\end{align}
We have 
\begin{align*}
B[w_n-w_m,(-f_n+f_m, v_n-v_m,)]&=
\|w_n-w_m\|_{E^{1}}^2
+E_{nm}, 
\end{align*}
where, using \eqref{estimate111}, $E_{nm}$ satisfies 
\begin{align}
|E_{nm}|\leq &\, 
C\|v_n-v_m\|_{L^2(\T^d)}\|D(v_n-v_m)\|_{L^2(\T^d)}
+C\|f_n-f_m\|_{L^2(\T^d)}\|v_n-v_m\|_{L^2(\T^d)}\nonumber\\
&
+C\|D(f_n-f_m)\|_{L^2(\T^d)}\|D(v_n-v_m)\|_{L^2(\T^d)}\nonumber\\
&+C\|f_n-f_m\|_{L^2(\T^d)}\|D(f_n-f_m)\|_{L^2(\T^d)}\nonumber\\
\leq &\, 
C \|w_n-w_m\|_{E^1} \big( \|z_n-z_m\|_{E^1}^2
+
\|z_n-z_m\|_{E^1}\|w_n-w_m\|_{E^1}\big)^{1/2}. \label{estimate112}
\end{align}
On the other hand, by Lemma \ref{lem:3-2} we have
\begin{equation}\label{estimate113}
B[w_n-w_m,(-f_n+f_m, v_n-v_m,)]\leq C\|z_n-z_m\|_{E^{1}}\|w_n-w_m\|_{E^{1}}. 
\end{equation}
Combining \eqref{estimate112} and \eqref{estimate113} we deduce 
\begin{align*}
&\|w_n-w_m\|_{E^1}^2\\
\leq &\, 
C\|z_n-z_m\|_{E^1}\|w_n-w_m\|_{E^1}\\
&+C \|w_n-w_m\|_{E^1}\big( \|z_n-z_m\|_{E^1}^2+
\|z_n-z_m\|_{E^1}\|w_n-w_m\|_{E^1}\big)^{1/2}. 
\end{align*}
By using Young's inequality we conclude
\[
\|w_n-w_m\|_{E^1}^{2}
\leq C\|z_n-z_m\|_{E^1}^{2}.
\]
From this we get convergence in $E^1$.  

Finally, we prove that $R(A)=E^1$.
Suppose that $R(A)\not=E^1$. Since $R(A)$ is closed, 
there would exist $z\in R(A)^\perp$ with $z\not=0$ such that  
$B_{\lam_0}[z,z]=0$. 
The argument in the proof of Lemma \ref{lem:3-3} implies $z=0$ which is a contradiction. 
\end{proof}

\begin{lem}\label{lem:3-5}
The operator $\cL_{\lam_0}:E^{k+2}\to E^k$ is an isomorphism
for all $k\in\N$ with $k\ge2$. 
\end{lem}
\begin{proof}
Since $\cL_{\lam_0}$ is injective, it suffices to prove that it is surjective. To do so, 
fix $w_0\in E^k$ with $w_0=(v_0,f_0)$. 
We claim there exists a solution $w_1\in E^{k+2}$
to $\cL_{\lam_0}w_1=w_0$. 

Consider the bounded  linear functional $w\mapsto(w_0,w)_{E^{0}}$ in $E^1$. 
By the Riesz representation theorem, there exists $\tilde{w}\in E^1$ such 
that $(w_0,w)_{E^{0}}=(\tilde{w},w)_{E^{1}}$ for any $w\in E^{1}$. 
In light of Lemmas \ref{lem:3-3} and \ref{lem:3-4}, there exists the inverse of $A$. We  
define 
 $w_1:=A^{-1}\tilde{w}$, and write $w_1=(v,f)$. 
Set
 \[
  \left(
  \begin{array}{l}
  \displaystyle
 g_1[v,f]\\
 g_2[v,f]
 \end{array}
  \right)   
 :=
 \left(
 \begin{array}{l}
 \displaystyle
 \frac{Du_{\lam_0}\cdot Dv}{m_{\lam_0}^{\al}}
 -\frac{\al |Du_{\lam_0}|^2 f}{2m_{\lam_0}^{\al+1}}
 +{\lam_0}b\cdot Dv-\big(\lam_0D_mV+(1-\lam_0)D_mV_0\big)f\\
 -\Div\big(m_{\lam_0}^{1-\al}Dv\big)
 -(1-\al)\Div\big(m_{\lam_0}^{-\al}f Du_{\lam_0}\big)-{\lam_0} \Div( b f)
 \end{array}
 \right).   
\] 
 Then, the identity
\[
(Aw_1,w)_{E^1}=(\tilde{w},w)_{E^1}=(w_0,w)_{E^0}
\]
for any $w\in E^1$,  
means that $v$ is a weak $H^1(\T^d)$ solution to
\[
v-\Delta v=g_1[v,f]+v_0,
\]
and that $f\in H^1(\T^d)$ is also a weak solution to
\[
f-\Delta f=g_2[v,f]+f_0.  
\]
Observe that if $v,f\in H^{j+1}(\Tt^d)$ then $g_1, g_2\in H^j(\T^d)$. 
Additionally, elliptic regularity yields, from $g_i[v,f]\in H^j(\T^d)$, that $v, j\in H^{j+2}(\T^d)$. 
Since we have $v,f\in H^1(\T^d)$, we conclude by induction that $v, f\in H^{j+2}(\T^d)$, for all $j\leq k$. 
\end{proof}

A straightforward result of Lemma \ref{lem:3-5} and the implicit function theorem 
in Banach space is 
\begin{prop}\label{prop:3-6}
The set $\Lam$ is relatively open in $[0,1]$. 
\end{prop}

We finally address the existence of solutions to \eqref{mfg-1}, \eqref{mfg-2},
and complete the proof of Theorem \ref{thm:main}.  
\begin{proof}[Proof of Theorem {\rm\ref{thm:main}}]
If $V$ is strictly increasing on $m$,  
the existence of a classical solution to \eqref{mfg-1}, \eqref{mfg-2} is a straightforward result of Proposition \ref{prop:3-6}. 
Uniqueness of solution is discussed in the next appendix, Proposition \ref{propuniq}.

If we only assume $V$ to be nondecreasing on $m$, 
existence can be obtained by using a perturbation argument similar to the one in \eqref{conti-1}. 
More precisely, we add a small perturbation $\ep\arctan(m)$ to $V$ so that we make the potential term strictly monotone. 
This problem admits a unique classical solution $(u_\epsilon, m_\epsilon)$.
Because the a-priori bounds in the previous section do not depend on the strict monotonicity of $V$, 
$(u_\epsilon, m_\epsilon)$
satisfy uniform bounds in any Sobolev space. 
Thus, by compactness, we can extract a convergent subsequence to a limit $(u,m)$ which solves \eqref{mfg-1}-\eqref{mfg-2}.
\end{proof}

\begin{rem}
In this paper, we focus on the case where $V$ is bounded on $m$, 
since our main concern is a lower bound on $m$.
In principle, unbounded potentials $V$ can be studied by adapting 
the techniques in \cite{ll3, CLLP, GPM1, GPM2}, for instance. 
\end{rem}

\appendix
\section{Uniqueness}\label{sec:unique}
Uniqueness of solutions of \eqref{mfg-1}-\eqref{mfg-2} is well understood (see \cite{ll2,GueU} for a related problem). However, 
to make this paper self-contained, 
we give a proof based on Lions ideas in \cite{LCDF}. 
%, as described in \cite[Theorem 2]{GueU}. 
%The proof also enlightens the importance of the range of $\al$ for the uniqueness. 

\begin{prop}
\label{propuniq}
The system \eqref{mfg-1}-\eqref{mfg-2} admits at most one classical solution $(u,m)$. 
\end{prop}
\begin{proof}
Let $(u_0,m_0)$ and $(u_1,m_1)$ be classical solutions to  \eqref{mfg-1}-\eqref{mfg-2}. 
Subtract 
\eqref{mfg-1} for $(u_1,m_1)$ from \eqref{mfg-1} for $(u_0,m_0)$, and 
\eqref{mfg-2} for $(u_1,m_1)$ from \eqref{mfg-2} for $(u_0,m_0)$, respectively, 
and then 
\begin{align}
u_0-u_1
&=\, 
\Del(u_0-u_1)+\frac{|Du_1|^2}{2m_{1}^{\al}}-\frac{|Du_0|^2}{2m_{0}^{\al}}
+b\cdot D(u_1-u_0)+V(x,m_0)-V(x,m_1), 
\label{unique-1}\\
m_0-m_1
&=\, 
\Del(m_0-m_1)+\Div(m_0^{1-\al}Du_0)-\Div(m_1^{1-\al}Du_1)+\Div(b(m_0-m_1)). 
\label{unique-2}
\end{align}
Subtract 
\eqref{unique-2} multiplied by $u_0-u_1$ from 
\eqref{unique-1} multiplied by $m_0-m_1$, and then  
\begin{align}
&\int_{\T^d}
\left(\frac{|Du_1|^2}{2m_{1}^{\al}}-\frac{|Du_0|^2}{2m_{0}^{\al}}\right)(m_0-m_1)\,dx
+
\int_{\T^d}
\left(m_0^{1-\al}Du_0-m_1^{1-\al}Du_1\right)\cdot D(u_0-u_1)\,dx\nonumber\\
&=\int_{\T^d}
(V(x,m_1)-V(x,m_0))(m_0-m_1)\,dx.\label{unique-3}
\end{align}

We prove that the left-hand side of \eqref{unique-3} is non-negative if $\al\in[0,2]$ following the technique in \cite{GueU}. 
Set $u_{\theta}:=u_0+\theta(u_1-u_0)$ 
and $m_{\theta}:=m_0+\theta(m_1-m_0)$ for $\theta\in[0,1]$. Define
\begin{multline*}
I(\theta):=\Big[
-\int_{\T^d}
\left(\frac{|Du_\theta|^2}{2m_{\theta}^\al}
-\frac{|Du_0|^2}{2m_{0}^\al}\right)(m_1-m_0)\\
+\int_{\T^d}\left(m_{\theta}^{1-\al}Du_{\theta}-m_{0}^{1-\al}Du_{0}\right)\cdot D(u_1-u_0)\,dx
\Big]. 
\end{multline*}
Then 
\begin{align*}
\frac{d}{dt}I(\theta)
=&\,
-\al\int_{\T^d}\frac{Du_{\theta}\cdot D(u_1-u_0)(m_1-m_0)}{m_{\theta}^\al}\,dx\\
&\,+
\frac{\al}{2}\int_{\T^d}
\frac{|Du_{\theta}|^2(m_1-m_0)^2}{m_{\theta}^{1+\al}}\,dx
+\int_{\T^d}m_{\theta}^{1-\al}|D(u_1-u_0)|^2\,dx\\
\ge&\,
\left(1-\frac{\al}{2}\right)\int_{\T^d}m_{\theta}^{1-\al}|D(u_1-u_0)|^2\,dx
\ge0 
\end{align*}
for $\al\in[0,2]$.  
Noting that $I(0)=0$, we conclude that $I(1)\ge 0$ which is the claim. 

Thus, by \eqref{unique-3} and the assumption that $V$ is strictly increasing on $m$,  
we get the conclusion. 
\end{proof}

\noindent
{\bf Acknowledgment.}
The authors would like to thank Hung V. Tran for suggestions and comments. 

%\bibliographystyle{plain}
%\bibliography{mfg}

%%%%%%%%%%%%%%%%%%%%%%%%%%%%%%%%%%%%%%%%%%%%%%%%%%%%%%%%%%%%%%%%%%%%%%%%%%%%%%%%%%%%%%%%%%

\end{document}